\documentclass{article}

\usepackage{amsmath,amsfonts,amssymb,amsthm,tikz}
\usepackage{algorithm}
\usepackage{algorithmic}
\usepackage[utf8]{inputenc}
\usepackage{cite}
\usepackage{enumitem}
\setlist[enumerate,1]{label=(\arabic*)} 

\usepackage{hyperref}

\newcommand{\C}{{\mathbb{C}}}

\newcommand{\Q}{{\mathbb{Q}}}

    \makeatletter
    \let\@fnsymbol\@alph
    \makeatother

\title{Monogenity of Fibonacci polynomials and Lucas polynomials}
\author{Han Chen, Weizhe Guo, Haojie Hong}
\date{\today}

\newtheorem{theorem}{Theorem}[section]
\newtheorem{proposition}[theorem]{Proposition}
\newtheorem{lemma}[theorem]{Lemma}
\newtheorem{corollary}[theorem]{Corollary}
\newtheorem{remark}[theorem]{Remark}
\newtheorem{problem}{Problem}

\numberwithin{equation}{section}

\makeatletter

\renewcommand*\l@section[2]{%
  \ifnum \c@tocdepth >\z@
    \addpenalty\@secpenalty
    \addvspace{0.2em \@plus\p@}%
    \setlength\@tempdima{1.5em}%
    \begingroup
      \parindent \z@ \rightskip \@pnumwidth
      \parfillskip -\@pnumwidth
      \leavevmode \bfseries
      \advance\leftskip\@tempdima
      \hskip -\leftskip
      #1\nobreak\hfil \nobreak\hb@xt@\@pnumwidth{\hss #2}\par
    \endgroup
  \fi}

\makeatother

\allowdisplaybreaks[1]

\begin{document}

\hfuzz 4.3pt

\maketitle

\renewcommand\thefootnote{}

\begin{abstract}
We investigate the monogenity of  irreducible factors of the Fibonacci polynomials $F_n(x)$ and the Lucas polynomials $L_n(x)$. 
Our main results show that for every odd positive integer $n$, all irreducible factors of $F_n(x)$ are monogenic, and for every even positive integer $n$, all irreducible factors of $L_n(x)$ are monogenic. 
\end{abstract}

{\footnotesize
	\tableofcontents}
\footnote{\textbf{Keywords.} \quad Fibonacci polynomials, Lucas polynomials, monogenic fields, cyclotomic fields, power integral bases.}
\footnote{\textbf{Mathematics Subject Classification (2020).}\quad 11R04, 11R18, 11R21, 11B39.}

\section{Introduction} 
Let $K$ be an algebraic number field of degree $n$ over $\mathbb{Q}$, and let $\mathcal{O}_K$ denote its 
ring of integers. We say that the field $K$ is monogenic, or that $\mathcal{O}_K$ possesses a power integral basis, if there exists an algebraic integer $\theta\in\mathcal{O}_K$ such that $\mathcal{O}_K=\mathbb{Z}[\theta]$. 
A non-zero irreducible polynomial $f(x)\in\mathbb{Z}[x]$ is said to be monogenic if there exists a root $\alpha$ of $f(x)$ such that $\mathbb{Z}[\alpha]$ is exactly the ring of integers of the field  $\mathbb{Q}(\alpha)$, i.e. $\mathbb{Q}(\alpha)$ is monogenic. Note that the  field $K$ may be monogenic without $f(x)$ being monogenic, which can be shown by considering the easy case that $K=\mathbb{Q}(\sqrt{5})$ and $f(x)=x^2-5$.

Research in monogenity can be  traced back to  Dedekind\cite{dedekind1878}, Hensel\cite{hensel1908}, Hasse\cite{hasse1963}, and has received considerable attention recently. For a comprehensive survey, we refer the reader to Ga\'{a}l's paper\cite{gaal2024}. While previous studies predominantly deal with statically defined polynomial families, there is little research on the monogenity of polynomial sequences. Very recently, Harrington and Jones \cite{harrington2025monogeniccyclicpolynomialsrecurrence} investigate the appearance of monogenic cyclic polynomials in certain linear recurrence sequences. Such dynamic polynomial families often exhibit rich arithmetic structures, making the characterization of their monogenicity a fascinating challenge.

In this note, we direct our attention to two of the most famous polynomials recurrence sequences, i.e. the Fibonacci and Lucas polynomials. The Fibonacci polynomial sequence $F_n(x)$ satisfies the recurrence relation 
\[
F_n(x) = xF_{n-1}(x) + F_{n-2}(x)
\]
with the initial terms $F_0(x) = 0$ and $F_1(x) = 1$, which can be viewed as a natural generalization of the Fibonacci numbers. Similarly, the Lucas polynomial sequence $L_n(x)$ satisfies the same recurrence relation but with the initial terms $L_0(x) = 2$ and $L_1(x) = x$. These two polynomial sequences are indexed in the On-Line Encyclopedia of Integer Sequences (OEIS) \cite{oeis} as A049310 and A114525, respectively.

Since the terms $F_n(x)$ and $L_n(x)$ are not generally irreducible over $\mathbb{Q}$, it is natural to investigate the arithmetic properties of their irreducible factors. We completely characterize the monogenity of the irreducible factors of $F_n(x)$ and $L_n(x)$ for specific parities of $n$. Precisely, we obtain the following main results.

 \begin{theorem}[]\label{monofibo}
For any odd positive integer $n$, all irreducible  factors of $F_{n}(x)$ are monogenic.
\end{theorem}

\begin{theorem}[]\label{monolucas}
	For any even positive integer $n$, all irreducible factors of $L_n(x)$
	are monogenic.
\end{theorem}


\begin{remark}
The parity conditions in Theorems \ref{monofibo} and \ref{monolucas} cannot be simply dropped to generalize the results to all positive integers $n$.  
	 For example, the odd-indexed Lucas polynomial $L_5(x)$ contains the irreducible factor $x^4+5x^2+5$, which is not monogenic. By the well-known identity $F_{2k}(x) = F_k(x)L_k(x)$, the same non-monogenic factor also appears in the even-indexed Fibonacci polynomial $F_{10}(x)$. Thus, the monogenity property does not universally hold for $F_n(x)$ when $n$ is even, nor for $L_n(x)$ when $n$ is odd. This justifies the parity restrictions in our main theorems and motivates the open problems presented at the end of this note.
\end{remark}

\paragraph{Plan of the article}
In 	Section 2, we first fix some notation and give some useful properties of polynomial recurrences. 

Our original proofs are motivated by some ingredients in \cite{harrington2025monogeniccyclicpolynomialsrecurrence}, where the authors use heavily the monogenity of cyclotomic subfield $\mathbb{Q}(\zeta+\zeta^{-1})$. In Section 3 and 4, we discuss the monogenity of fields $\mathbb{Q}(\zeta_{n}-\zeta_{n}^{-1})$ and $\mathbb{Q}(i(\zeta_{2n}+\zeta_{2n}^{-1}))$.
Theorem \ref{monofibo} and Theorem \ref{monolucas} are proved in Section 5 and 6, respectively. Finally, in Section 7, we list several open problems.

\section{Notation and Preliminaries}

Assume $n>1$ is a positive integer, 
let $\zeta_n=e^{2\pi i/n} \in \C$ be a primitive $n$-th root of unity. Note that when $n$ is odd,  $-(\zeta_n)^{n+1}=-\zeta_{n}$ is a primitive $2n$-th root of unity. We denote by $d_K$  the discriminant of an algebraic number field $K$. 

In the following, we give explicit formulas and the exact roots of Fibonacci polynomials and Lucas polynomials, which allows us to investigate irreducible factors of their terms. 
Although we use more refined discussions, item \ref{zerofnpoly} of Lemma \ref{propertiesfnpoly} and item \ref{zerolucaspoly} of Lemma \ref{propertieslucaspoly} can essentially 
be found in \cite{1973zero}. 
The rest items are quite well-known and can be found in \cite{koshy2019}. For the reader's convenience, we include full proofs in this section.

	\begin{lemma}[Properties of Fibonacci polynomials]\label{propertiesfnpoly}
	Let $F_n(x)$ be the Fibonacci polynomial. Then
	\begin{enumerate}
		\item\label{binetoffnpoly} The Fibonacci polynomials have the explicit formula
		\[
		F_n(x) = \frac{(x+\sqrt{x^2+4})^n - (x-\sqrt{x^2+4})^n}{2^n\sqrt{x^2+4}}, \quad n=0,1,2,\cdots.
		\]
		
		\item\label{zerofnpoly} 
		The set of zeros of $F_n(x)$ is
		\[
		R = \left\{ 2i\cos \left(k\pi/n\right) : k=1,2,\cdots,n-1 \right\}.
		\]
	\end{enumerate}
\end{lemma}

\begin{proof}
	(1) The Fibonacci polynomials satisfy the recurrence relation $F_n(x) = xF_{n-1}(x) + F_{n-2}(x)$ with the initial terms $F_0(x) = 0$ and $F_1(x) = 1$. 
	The characteristic equation corresponding to this linear recurrence sequence is given by
	\[
	\lambda^2 - x\lambda - 1 = 0.
	\]
	Solving this quadratic equation yields two distinct characteristic roots 
	\[
	\alpha = \frac{x + \sqrt{x^2+4}}{2}, \quad \beta = \frac{x - \sqrt{x^2+4}}{2}.
	\]
	The general solution for $F_n(x)$ can be expressed as a linear combination of these characteristic roots, say 
	\[
	F_n(x) = c_1\alpha^n + c_2\beta^n,
	\]
	where $c_1$ and $c_2$ are constants determined by the initial conditions. 
	Using the initial values $F_0(x) = 0$ and $F_1(x) = 1$, we obtain 
	\[
	\begin{cases}
		c_1 + c_2 = 0, \\
		c_1\alpha + c_2\beta = 1.
	\end{cases}
	\]
	Since $\alpha - \beta = \sqrt{x^2+4}$, we find
	 $
	c_1 = \frac{1}{\sqrt{x^2+4}}, \quad c_2 = -\frac{1}{\sqrt{x^2+4}}.
	$  
	This yields the Binet-like explicit formula
	\[
	F_n(x) = \frac{\alpha^n - \beta^n}{\alpha - \beta} = \frac{(x+\sqrt{x^2+4})^n - (x-\sqrt{x^2+4})^n}{2^n\sqrt{x^2+4}}.
	\]
	
	(2) To determine the zeros of $F_n(x)$, we set $F_n(x) = 0$. From the explicit formula derived in part (1), this condition is equivalent to
	$
\alpha^n = \beta^n.
	$ 
	By Vieta's formulas for the characteristic equation, we have $\alpha\beta = -1$, which implies $\beta = -\alpha^{-1}$. Substituting this into the equation yields
	\[
 \alpha^{2n} = (-1)^n.
	\]
	This implies $(-\alpha^2)^n = 1$, meaning $-\alpha^2$ is an $n$-th root of unity. Hence, we can write
	\[
	-\alpha^2 = \zeta_{n}^k
	\]
	for some $k$. 
	Taking the square root, we obtain 
	$
	\alpha = \pm i \zeta_{2n}^k.
	$
	Using the relationship $x = \alpha + \beta = \alpha - \alpha^{-1}$, we have
	\[
\begin{aligned}
x &= \pm i \zeta_{2n}^k - (\pm i \zeta_{2n}^k)^{-1} =  \pm 2i \cos\left(k\pi/n\right).
\end{aligned}
\]
	Since $\cos(\pi - \theta) = -\cos(\theta)$, the roots obtained from the negative sign are identical to the set of roots obtained from the positive sign under a reindexing of $k$. 
	Because the cosine function $\cos(k\pi/n)$ is strictly decreasing and distinct on the interval $(0, \pi)$ for $k=1,2,\cdots,n-1$, we obtain $n-1$ distinct roots:
	\[
	x_k = 2i\cos\left(k\pi/n\right), \quad k=1,2,\cdots,n-1.
	\]
	From the recurrence relation, it is evident that $F_n(x)$ is a polynomial of degree $n-1$ for $n \ge 1$. By the Fundamental Theorem of Algebra, the set $R$ constitutes the complete set of zeros of $F_n(x)$.
\end{proof}

\begin{lemma}\label{propertieslucaspoly}
	Let $F_n(x)$ and $L_n(x)$ denote the Fibonacci and Lucas polynomials, respectively. Then
	\begin{enumerate}
		\item\label{Lucaswithfngl} $L_n(x) = F_{n-1}(x) + F_{n+1}(x)$ for $n \ge 1$.
		
		\item\label{binetolucaspoly} The Lucas polynomials satisfy the explicit formula
		\[
		L_n(x) = 2^{-n} \left[ (x+\sqrt{x^2+4})^n + (x-\sqrt{x^2+4})^n \right].
		\]
		
		\item \label{zerolucaspoly} The set $R$ of zeros of $L_n(x)$ is given by
		\[
		R = \begin{cases}
			\left\{ 2i\sin\left(\frac{k\pi}{n}\right) : k = -\frac{n-1}{2}, \dots, \frac{n-1}{2} \right\}, & \text{if } n \text{ is odd,} \\[1ex]
			\left\{ 2i\sin\left(\frac{(2k+1)\pi}{2n}\right) : k = -\frac{n}{2}, \dots, \frac{n}{2}-1 \right\}, & \text{if } n \text{ is even.}
		\end{cases}
		\]
	\end{enumerate}
\end{lemma}

\begin{proof}
	(1) We proceed by strong induction on $n$. 
	For the base case $n=1$, we have $F_0(x) + F_2(x) = 0 + x = x = L_1(x)$. Assume the statement holds for all integers up to $k \ge 1$. For $n = k+1$, we have
	\[
	\begin{aligned}
		L_{k+1}(x) &= xL_k(x) + L_{k-1}(x) \\
		&= x(F_{k-1}(x) + F_{k+1}(x)) + (F_{k-2}(x) + F_{k}(x)) \\
		&= (xF_{k-1}(x) + F_{k-2}(x)) + (xF_{k+1}(x) + F_k(x)) \\
		&= F_k(x) + F_{k+2}(x).
	\end{aligned}
	\]
	By induction, the formula holds for all $n \ge 1$.
	
	(2) Let $\alpha = \frac{x+\sqrt{x^2+4}}{2}$ and $\beta = \frac{x-\sqrt{x^2+4}}{2}$. Notice that $\alpha - \beta = \sqrt{x^2+4}$,  $\alpha\beta = -1$ and 
	 $F_n(x) = \frac{\alpha^n - \beta^n}{\alpha - \beta}$. By part (1), we obtain
	\[
	\begin{aligned}
		L_n(x) &= F_{n-1}(x) + F_{n+1}(x) \\
		&= \frac{\alpha^{n-1} - \beta^{n-1}}{\alpha - \beta} + \frac{\alpha^{n+1} - \beta^{n+1}}{\alpha - \beta} \\
		&= \frac{\alpha^{n-1}(1 + \alpha^2) - \beta^{n-1}(1 + \beta^2)}{\alpha - \beta}.
	\end{aligned}
	\]
	Using the relation $\alpha\beta = -1$, we can deduce $1 + \alpha^2 = -\alpha\beta + \alpha^2 = \alpha(\alpha - \beta)$. Similarly, $1 + \beta^2 = -\alpha\beta + \beta^2 = -\beta(\alpha - \beta)$. Substituting these identities into the equation yields
	\[
	\begin{aligned}
		L_n(x) &= \frac{\alpha^{n-1} \cdot \alpha(\alpha - \beta) - \beta^{n-1} \cdot (-\beta(\alpha - \beta))}{\alpha - \beta} \\
		&= \frac{(\alpha^n + \beta^n)(\alpha - \beta)}{\alpha - \beta} \\
		&= \alpha^n + \beta^n 
		,
	\end{aligned}
	\]
	which is exactly the desired formula.
	
	(3) From  part (2), we already established that $L_n(x) = \alpha^n + \beta^n$, with $\alpha\beta = -1$ and $x = \alpha + \beta$. Setting $L_n(x) = 0$ yields
	\[
	\alpha^n + \beta^n = 0 \implies \alpha^{2n} = (-1)^{n+1}.
	\]
	
	If $n$ is even, then $\alpha^{2n} = -1$. This equation has $2n$ distinct roots of the form $\alpha = \zeta_{4n}^{2k+1}$ for $k = 0, 1, \dots, 2n-1$. 
	Since $\beta = -1/\alpha$, we have $x = \alpha + \beta = \alpha - \alpha^{-1}$. Substituting the roots gives
	\[
	x = \zeta_{4n}^{2k+1} - \zeta_{4n}^{-(2k+1)} = 2i\sin\left(\frac{(2k+1)\pi}{2n}\right).
	\]
	Due to the symmetry $\sin(\pi - \theta) = \sin(\theta)$, the mapping from $\alpha$ to $x$ is 2-to-1. To obtain exactly $n$ distinct zeros of $L_n(x)$, we restrict the index $k$ to a subset that yields unique sine values, namely $k = -\frac{n}{2}, \dots, \frac{n}{2}-1$. 
	
	If $n$ is odd, then $\alpha^{2n} = 1$. This equation has $2n$ distinct roots of the form $\alpha = \zeta_{2n}^k$ for $k = 0, 1, \dots, 2n-1$. 
	Similarly, $x = \alpha - \alpha^{-1} = 2i\sin\left(\frac{k\pi}{n}\right)$. 
	By the same symmetry property of the sine function, we restrict $k$ to a symmetric range around zero to obtain the $n$ distinct roots, namely $k = -\frac{n-1}{2}, \dots, \frac{n-1}{2}$.
\end{proof}

\section{Monogenity of $\mathbb{Q}(\zeta_{n}-\zeta_{n}^{-1})$}
In this section, we investigate the monogenity of the cyclotomic subfield $\Q(\zeta_{n}-\zeta_{n}^{-1})$, where $n>1$ is a positive integer.
We begin by determining the  extension degree of the cyclotomic field $\Q(\zeta_{n})$ over this subfield through a Galois-theoretic approach.

\begin{lemma}\label{n4bsdegreeofsubfield1}
	Let $n>1$ be a positive integer. If $4 \mid n$ and $n > 4$, then $[\mathbb{Q}(\zeta_{n}) : \mathbb{Q}(\zeta_{n}-\zeta_{n}^{-1})] = 2$. If $4 \nmid n$ or $n = 4$, then $\mathbb{Q}(\zeta_{n}) = \mathbb{Q}(\zeta_{n}-\zeta_{n}^{-1})$.
\end{lemma}

\begin{proof}
	Since $\zeta_{n}$ satisfies $x^2 - (\zeta_{n}-\zeta_{n}^{-1})x - 1 = 0$, it follows that $[\mathbb{Q}(\zeta_{n}) : \mathbb{Q}(\zeta_{n}-\zeta_{n}^{-1})] \le 2$.
	
	The extension $\mathbb{Q}(\zeta_{n}) / \mathbb{Q}$ is Galois, and $\mathrm{Gal}(\mathbb{Q}(\zeta_{n})/\mathbb{Q}) \cong (\mathbb{Z}/n\mathbb{Z})^{*}$. For any $a \in (\mathbb{Z}/n\mathbb{Z})^{*}$, there exists $\sigma_{a} \in \mathrm{Gal}(\mathbb{Q}(\zeta_{n})/\mathbb{Q})$ such that $\sigma_{a}(\zeta_{n}) = \zeta_{n}^a$. Let 
	\[
	H = \{ \sigma_{a} \in \mathrm{Gal}(\mathbb{Q}(\zeta_{n})/\mathbb{Q}) \mid \sigma_{a}(\zeta_{n}-\zeta_{n}^{-1}) = \zeta_{n}-\zeta_{n}^{-1} \}
	\]
	be a subgroup of $\mathrm{Gal}(\mathbb{Q}(\zeta_{n})/\mathbb{Q})$. The fixed field of $H$ is exactly $\mathbb{Q}(\zeta_{n}-\zeta_{n}^{-1})$, which implies $|H| = [\mathbb{Q}(\zeta_{n}) : \mathbb{Q}(\zeta_{n}-\zeta_{n}^{-1})]$.
	
	For $\sigma_{a} \in \mathrm{Gal}(\mathbb{Q}(\zeta_{n})/\mathbb{Q})$, we have $\sigma_{a}(\zeta_{n}-\zeta_{n}^{-1}) = \zeta_{n}^a - \zeta_{n}^{-a}$.
	If $\sigma_{a} \in H$, then $\zeta_{n}^a - \zeta_{n}^{-a} = \zeta_{n} - \zeta_{n}^{-1}$. Multiplying both sides by $\zeta_{n}^{a+1}$ yields
	\[
	\begin{aligned}
		\zeta_{n}^{2a+1} - \zeta_{n} - \zeta_{n}^{a+2} + \zeta_{n}^{a} &= 0 \\
		(\zeta_{n}^a - \zeta_{n})(\zeta_{n}^{a+1} + 1) &= 0 \\
	\end{aligned}
	\]
	When $a \equiv 1 \pmod{n}$, we get the identity element $\sigma_{1} = 1$.

	If $n$ is odd, then the equation $\zeta_n^{a+1}+1=0$ has no solution. When $n$ is even, if $a$ satisfies $(\zeta_n^a-\zeta_n)(\zeta_n^{a+1}+1)=0$, then $a \equiv 1 \pmod{n}$ or $a \equiv \frac{n}{2}-1 \pmod{n}$.
	    
	If $n \equiv 2 \pmod{4}$, let $n = 2m$ where $m$ is odd. 
	Since $2 \mid (m-1, 2m)$, any $a$ satisfying $a \equiv \frac{n}{2}-1 \pmod{n}$ is not coprime to $n$. Thus $a \notin (\mathbb{Z}/n\mathbb{Z})^{*}$, meaning $\sigma_{a} \notin \mathrm{Gal}(\mathbb{Q}(\zeta_{n})/\mathbb{Q})$.
	If $n = 4$, then $\frac{n}{2}-1 = 1$, which is the same as $a \equiv 1 \pmod{n}$.
	Therefore, if $4 \nmid n$ or $n = 4$, $H = \{1\}$, and $|H| = [\mathbb{Q}(\zeta_{n}) : \mathbb{Q}(\zeta_{n}-\zeta_{n}^{-1})] = 1$, yielding $\mathbb{Q}(\zeta_{n}) = \mathbb{Q}(\zeta_{n}-\zeta_{n}^{-1})$.
	
	If $4 \mid n$ and $n > 4$, let $n = 4k$. Then $\frac{n}{2}-1 = 2k-1$ is odd. Since $2k-1 > 1$ and $(2k-1, n) = 1$, the solution $a \equiv \frac{n}{2}-1 \pmod{n}$ is distinct from $a \equiv 1 \pmod{n}$, and such $a$ belongs to $(\mathbb{Z}/n\mathbb{Z})^{*}$. Thus $\sigma_{a} \ne 1$ and $\sigma_{a} \in H$. Since $|H| = [\mathbb{Q}(\zeta_{n}) : \mathbb{Q}(\zeta_{n}-\zeta_{n}^{-1})] \le 2$, it must be that $H = \{1, \sigma_{a}\}$, which concludes $[\mathbb{Q}(\zeta_{n}) : \mathbb{Q}(\zeta_{n}-\zeta_{n}^{-1})] = 2$.
\end{proof}

\begin{proposition}\label{comfieldmonolucas}
	Let $n$ be an integer. If $4 \mid n$ and $n > 4$, then $\mathbb{Z}[\zeta_{n}-\zeta_{n}^{-1}]$ is the ring of integers of $\mathbb{Q}(\zeta_{n}-\zeta_{n}^{-1})$. In particular, $\mathbb{Q}(\zeta_{n}-\zeta_{n}^{-1})$ is monogenic.
\end{proposition}

\begin{remark}
 Nakahara and Shah \cite{NakaharaMonogenesis2002} showed that $\mathbb{Z}[\zeta_m-\zeta_m^{-1}]$ is the full ring of integers for $m = 2^n \ge 8$ or $m = 4p^n$ with $p$ an odd prime. Our Proposition \ref{comfieldmonolucas} extends this to every $m$ divisible by $4$ and greater than $4$, using a different method.
\end{remark}
\begin{proof}
	Assume $4 \mid n$ and $n > 4$. By Lemma \ref{n4bsdegreeofsubfield1}, $[\mathbb{Q}(\zeta_{n}) : \mathbb{Q}(\zeta_{n}-\zeta_{n}^{-1})] = 2$. Suppose $[\mathbb{Q}(\zeta_{n}-\zeta_{n}^{-1}) : \mathbb{Q}] = t$, then $t = \frac{1}{2}\varphi(n)$.
	
	Suppose for the sake of contradiction that $\mathbb{Z}[\zeta_{n}-\zeta_{n}^{-1}]$ is not the full ring of integers. Then there exists an algebraic integer $\alpha \in \mathcal{O}_{\mathbb{Q}(\zeta_{n}-\zeta_{n}^{-1})}$ with rational, non-integer coefficients in the basis $\{1, \zeta_{n}-\zeta_{n}^{-1}, \dots, (\zeta_{n}-\zeta_{n}^{-1})^{t-1}\}$, say \[
	\alpha = \sum_{k=0}^{t-1} a_k (\zeta_n-\zeta_n^{-1})^k.
	\]
	
	Let $N \le t - 1 = \frac{1}{2}\varphi(n) - 1$ be the largest index such that its coefficient $a_N \notin \mathbb{Z}$. By subtracting $\sum_{k=N+1}^{t-1} a_k(\zeta_{n}-\zeta_{n}^{-1})^k \in \mathbb{Z}[\zeta_{n}-\zeta_{n}^{-1}]$ from $\alpha$, we may assume without loss of generality that
	\[
	\alpha = a_0 + a_1(\zeta_{n}-\zeta_{n}^{-1}) + \cdots + a_N(\zeta_{n}-\zeta_{n}^{-1})^N
	\]
	where $a_N \notin \mathbb{Z}$ and $a_i \in \mathbb{Q}$ for $0 \le i \le N$. Note that $\alpha$ is also an algebraic integer in $\mathbb{Q}(\zeta_{n})$. 
	
	Multiplying by $\zeta_n^N$ and expanding the result as a polynomial in $\zeta_n$, we obtain
	\[
	\begin{aligned}
		\zeta_n^N \alpha &= \sum_{k=0}^{N} a_{k} \zeta_{n}^N \sum_{j=0}^{k} \binom{k}{j} \zeta_{n}^{k-j}(-\zeta_{n}^{-1})^j \\
		&= \sum_{k=0}^{N} a_{k} \sum_{j=0}^{k} (-1)^j \binom{k}{j} \zeta_{n}^{N+k-2j} \\
		&= a_{N} \sum_{j=0}^{N} (-1)^j \binom{N}{j} \zeta_{n}^{2N-2j} + \sum_{k=0}^{N-1} a_{k} \sum_{j=0}^{k} (-1)^j \binom{k}{j} \zeta_{n}^{N+k-2j} 
		.
	\end{aligned}
	\]
	Since both $\zeta_n$ and $\alpha$ are algebraic integers, their product $\zeta_n^N \alpha \in \mathbb{Z}[\zeta_n]$. Because $N \leq \frac{1}{2}\varphi(n) - 1$, the highest exponent of $\zeta_n$ in this expansion is $2N \leq \varphi(n) - 2 < \varphi(n)$. 
	
	The set $\{1, \zeta_n, \dots, \zeta_n^{\varphi(n)-1}\}$ forms an integral basis for $\mathbb{Q}(\zeta_n)$. Since all powers of $\zeta_n$ in our expansion are strictly less than $\varphi(n)$, the representation of $\zeta_n^N \alpha$ in this basis is already unique. 
	 This strictly forces the leading coefficient $a_N$ to be in $\mathbb{Z}$, which contradicts our assumption.
	
	Thus, all coefficients must be integers. The set $\{1, \zeta_{n}-\zeta_{n}^{-1}, \dots, (\zeta_{n}-\zeta_{n}^{-1})^{t-1}\}$ is an integral basis, meaning $\mathbb{Z}[\zeta_{n}-\zeta_{n}^{-1}]$ is the ring of integers of $\mathbb{Q}(\zeta_{n}-\zeta_{n}^{-1})$, making it monogenic.
\end{proof}

\section{Monogenity of $\mathbb{Q}(i(\zeta_{2n}+\zeta_{2n}^{-1}))$}

In this section, we establish the following result.
\begin{proposition}\label{comfieldmono}
	Assume $n$ is an odd positive integer, $\mathbb{Z}[i(\zeta_{2n}+\zeta_{2n}^{-1})]$ is the ring of integers of $\mathbb{Q}(i(\zeta_{2n}+\zeta_{2n}^{-1}))$. In particular,
	the field $\mathbb{Q}(i(\zeta_{2n}+\zeta_{2n}^{-1}))$ is monogenic.
\end{proposition}

\begin{remark} Proposition \ref{comfieldmono} was already proved by Motoda, Nakahara and Shah \cite{motodaProblemHasseCertain2002} by showing that the norm of $i(\zeta_{2n}+\zeta_{2n}^{-1})$ coincides with the discriminant of the field. We use a different approach to prove that the ring of integers is directly generated by $i(\zeta_{2n}+\zeta_{2n}^{-1})$. 
\end{remark}
The case $n=1$ is trivial, since $\mathbb{Q}(i(\zeta_2+\zeta_2^{-1})) = \mathbb{Q}(i)$, whose ring of integers is precisely $\mathbb{Z}[i]$. Hence, in the following preparation, we assume $n > 1$.

Since $n$ is odd, we can represent the primitive $2n$-th root of unity as $\zeta_{2n} = -\zeta_n$. Then
\[
\zeta_{2n} + \zeta_{2n}^{-1} = -\zeta_n - \zeta_n^{-1} = -(\zeta_n + \zeta_n^{-1}).
\]
Therefore, $i(\zeta_{2n} + \zeta_{2n}^{-1}) = -i(\zeta_n + \zeta_n^{-1})$. 
Thus, it suffices to prove that the ring of integers of $K := \mathbb{Q}(i(\zeta_n + \zeta_n^{-1}))$ is $\mathbb{Z}[i(\zeta_n + \zeta_n^{-1})]$. We need some lemmas as follows.

\begin{lemma}\label{lem5}
	For all $k\geq 0$, we have $\zeta_n^{2k}+\zeta_n^{-2k}\in \mathbb{Z}[i(\zeta_n+\zeta_n^{-1})]$.
\end{lemma}
\begin{proof}
	Let $A = \mathbb{Z}[i(\zeta_n+\zeta_n^{-1})]$. The result is trivial when $k=0$.
	For $k=1$, note that 
	\[ [i(\zeta_n+\zeta_n^{-1})]^2 = -(\zeta_n^2+\zeta_n^{-2}+2), \]
	which implies $\zeta_n^2+\zeta_n^{-2} \in A$.
	Assume the assertion holds for all $k < j$, where $j\geq 2$. Then:
	\[ \zeta_n^{2j}+\zeta_n^{-2j} = (\zeta_n^{2(j-1)}+\zeta_n^{-2(j-1)})(\zeta_n^2+\zeta_n^{-2}) - (\zeta_n^{2(j-2)}+\zeta_n^{-2(j-2)}). \]
	By the induction hypothesis, all terms on the right-hand side belong to $A$, hence $\zeta_n^{2j}+\zeta_n^{-2j} \in A$. By induction, Lemma \ref{lem5} holds for all $k\geq 0$.
\end{proof}

\begin{corollary}\label{cor1}
	For all $k\in \mathbb{Z}$, we have $\zeta_n^{k}+\zeta_n^{-k}\in \mathbb{Z}[i(\zeta_n+\zeta_n^{-1})]$.
\end{corollary}
\begin{proof}
	Since $n$ is odd, $2$ is invertible modulo $n$. For any $t \in \mathbb{Z}$, there exists an integer $t_1$ such that $t \equiv 2t_1 \pmod{n}$. Thus, $\zeta_n^t+\zeta_n^{-t} = \zeta_n^{2t_1}+\zeta_n^{-2t_1}$. Replacing $t_1$ by $-t_1$ if necessary, we may assume $t_1 \ge 0$, and the conclusion follows directly from Lemma \ref{lem5}.
\end{proof}

\begin{lemma}\label{lem7}
	For all $k\geq 0$, we have $i(\zeta_n^{2k+1}+\zeta_n^{-(2k+1)}) \in \mathbb{Z}[i(\zeta_n+\zeta_n^{-1})]$.
\end{lemma}
\begin{proof}
	Let $A = \mathbb{Z}[i(\zeta_n+\zeta_n^{-1})]$. By Lemma \ref{lem5}, $\zeta_n^2+\zeta_n^{-2} \in A$.
	The case $k=0$ is clear since $i(\zeta_n+\zeta_n^{-1}) \in A$.
	For $k=1$, we have:
	\[ i(\zeta_n^3+\zeta_n^{-3}) = i(\zeta_n+\zeta_n^{-1})(\zeta_n^2+\zeta_n^{-2}) - i(\zeta_n+\zeta_n^{-1}) \in A. \]
	Assume the assertion holds for all $0 \leq k \leq j-1$ for $j \geq 2$. Then:
	\[ i(\zeta_n^{2j+1}+\zeta_n^{-(2j+1)}) = i(\zeta_n^{2j-1}+\zeta_n^{-(2j-1)})(\zeta_n^2+\zeta_n^{-2}) - i(\zeta_n^{2j-3}+\zeta_n^{-(2j-3)}). \] 
	Both terms on the right-hand side belong to $A$, so the left-hand side also belongs to $A$. This proves the lemma by induction.
\end{proof}

\begin{corollary}\label{cor2}
	For all $k\in \mathbb{Z}$, we have $i(\zeta_n^{k}+\zeta_n^{-k})\in \mathbb{Z}[i(\zeta_n+\zeta_n^{-1})]$.
\end{corollary}
\begin{proof}
	Since $n$ is odd, every residue class modulo $n$ can be represented by an odd integer. Hence, for every $k \in \mathbb{Z}$, there exists an odd integer $m$ such that $k \equiv m \pmod{n}$. Thus, $\zeta_n^k+\zeta_n^{-k} = \zeta_n^m+\zeta_n^{-m}$. Replacing $m$ by $-m$ if necessary, we may assume $m > 0$. Then $m = 2r+1$ for some $r \geq 0$, and the conclusion follows from Lemma \ref{lem7}.
\end{proof}

For convenience, let $M,N,$ and $L$ denote the  fields $\mathbb{Q}(i), \mathbb{Q}(\zeta_n+\zeta_n^{-1}), $ and $ \mathbb{Q}(i, \zeta_n+\zeta_n^{-1})$, respectively.

\begin{proposition}\label{prop1}
	$K = L = MN$, i.e., $\mathbb{Q}(i(\zeta_n+\zeta_n^{-1})) = \mathbb{Q}(i, \zeta_n+\zeta_n^{-1})$.
\end{proposition}
\begin{proof}
	By Corollary \ref{cor1}, $\zeta_n+\zeta_n^{-1} \in \mathbb{Z}[i(\zeta_n+\zeta_n^{-1})] \subset K$, which means $N \subseteq K \subseteq L$. Since $N$ is totally real and $K \not\subseteq \mathbb{R}$ (as $i(\zeta_n+\zeta_n^{-1})$ is purely imaginary and non-zero for $n > 1$), $N \neq K$. Given that $[L:N] = 2$, it forces $[L:K] = 1$, and so $K = L = MN$.
\end{proof}

\begin{lemma}\label{lem6}
	The ring of integers of $K$ is $\mathcal{O}_K = \mathbb{Z}[i, \zeta_n+\zeta_n^{-1}]$.
\end{lemma}
\begin{proof}
Note that the discriminant of $M$ is $d_M = -4$. For the field $N$, we know that $|d_N| \mid |d_{\mathbb{Q}(\zeta_n)}| \mid n^{\phi(n)}$ (see \cite[Theorem 2.31 and Corollary 7.11]{MR4442264}). Since $n$ is odd, $d_N$ only has odd prime factors, which yields  $\gcd(d_M, d_N) = 1$. By the theorem on the compositum of fields with coprime discriminants (see \cite[Corollary 2.28]{MR4442264}), $\mathcal{O}_K = \mathcal{O}_M \cdot \mathcal{O}_N = \mathbb{Z}[i] \cdot \mathbb{Z}[\zeta_n+\zeta_n^{-1}] = \mathbb{Z}[i, \zeta_n+\zeta_n^{-1}]$. The fact that $\mathcal{O}_N = \mathbb{Z}[\zeta_n+\zeta_n^{-1}]$ can be found, for instance, in \cite[Proposition 2.16]{washington97}.
\end{proof}

Now we are ready to prove the main result.

\begin{proof}[Proof of Proposition \ref{comfieldmono}]
	As discussed, it suffices to show $\mathcal{O}_K = \mathbb{Z}[i(\zeta_n+\zeta_n^{-1})]$ for $n>1$. Since $n > 1$ is odd, the primitive $n$-th root of unity $\zeta_n$ satisfies:
	\[ 1+\sum_{1 \leq k \leq (n-1)/2}(\zeta_n^k+\zeta_n^{-k}) = 0. \]
	Multiplying by $i$, we can express $-i$ as:
	\[ -i = \sum_{1 \leq k \leq (n-1)/2}i(\zeta_n^k+\zeta_n^{-k}). \]
	By Corollary \ref{cor2}, every term in this sum is in $\mathbb{Z}[i(\zeta_n+\zeta_n^{-1})]$. Therefore, $-i \in \mathbb{Z}[i(\zeta_n+\zeta_n^{-1})]$, which implies $\mathbb{Z}[i] \subseteq \mathbb{Z}[i(\zeta_n+\zeta_n^{-1})]$.
	
	Combining this with Corollary \ref{cor1}, we see that both generators $i$ and $\zeta_n+\zeta_n^{-1}$ are contained in $\mathbb{Z}[i(\zeta_n+\zeta_n^{-1})]$. Using Lemma \ref{lem6}, we establish:
	\[ \mathcal{O}_K = \mathbb{Z}[i, \zeta_n+\zeta_n^{-1}] \subseteq \mathbb{Z}[i(\zeta_n+\zeta_n^{-1})]. \]
	Since the reverse inclusion $\mathbb{Z}[i(\zeta_n+\zeta_n^{-1})] \subseteq \mathcal{O}_K$ is trivially true, we conclude that $\mathcal{O}_K = \mathbb{Z}[i(\zeta_n+\zeta_n^{-1})]$. This completes the proof.
\end{proof}

\section{Proof of Theorem \ref{monofibo}}
For each divisor $d>1$ of $n$, we define 
\begin{equation*}\label{}
R_d : = \{
i(\zeta_{2n}^k + \zeta_{2n}^{-k}): (k,n)=n/d,  k\in\{
1,2,\cdots,n-1
\}
\} .
\end{equation*}
Let $m:=kd/n$. 
For any $k$ such that $(k,n)=n/d$, we have $(m,d)=1$. Thus $|R_d| = \varphi(d)$ and $R_a\cap R_b = \varnothing$ for $a\ne b$. By item \ref{zerofnpoly} of 
Lemma \ref{propertiesfnpoly}, we know that $R = \cup_d R_d$. Define \[
\Omega_d(x) : = \prod_{x_i\in R_d} (x-x_i).
\]
Rewrite $R_d$ as follows
\[R_d=\{i(\zeta_{2d}^k + \zeta_{2d}^{-k}), \, \text{where}\, 1\leq k \leq d, (k,d)=1\}.
\]
By Galois theory, the set of all conjugates of $i(\zeta_{2d} + \zeta_{2d}^{-1})$ over $\mathbb{Q}$ is 
\[\mathcal{S}_d := \{\pm i(\zeta_{2d}^k + \zeta_{2d}^{-k}), \, \text{where}\,   1\leq k \leq d,  (k,2d)=1\}. \]
When $1 \leq k \leq d, (k,d)=1$ and $k$ is odd, it follows that 
\[(k, 2d)=1;\] 
When $1 \leq k \leq d, (k,d)=1$ and $k$ is even, we have
\[1\leq d-k\leq d, (d-k, 2d)=1\] 
and 
\[i(\zeta_{2d}^{-k} + \zeta_{2d}^{k})
=-i(\zeta_{2d}^{d-k} + \zeta_{2d}^{-(d-k)}).
\]

Therefore, we obtain 
\[R_d=\mathcal{S}_d.\]
Thus $\Omega_d(x)= \prod_{x_i\in R_d} (x-x_i)$ is the minimal polynomial of the algebraic integer $i(\zeta_{2d} + \zeta_{2d}^{-1})$ over $\mathbb{Q}$ and $\Omega_d(x) \in \mathbb{Z}[x]$ is irreducible over $\mathbb{Q}$ by the characteristic of the minimal polynomial.

Now $\Omega_d(x)$ is monogenic by Proposition \ref{comfieldmono}.


\section{Proof of Theorem \ref{monolucas}}
Let $n$ be an even positive integer. For each odd divisor $g$ of $n$, we define $d = 2n/g$. 
Note that since $n$ is even and $g$ is odd, $d$ is always a multiple of $4$. 
We define 
\[
R_d := \left\{ \zeta_{4n}^{2k+1} - \zeta_{4n}^{-(2k+1)} : (2k+1, 4n) = 2n/d,  k \in \{-n/2, \dots, n/2-1\} \right\}.
\]
Let $z = \frac{(2k+1)d}{2n}$. For any $k$ such that $(2k+1, 4n) = 2n/d$, we should have $(z, 2d)=1$. Thus $|R_d| = \varphi(2d)/2$ and $R_a \cap R_b = \varnothing$ for $a \ne b$. 

By item \ref{zerolucaspoly} of Lemma \ref{propertieslucaspoly}, we know that $R = \cup_{d} R_d$. Define 
\[
\Omega_d(x) := \prod_{x_i \in R_d} (x-x_i).
\]
Rewrite $R_d$ as follows
\[
R_d = \left\{ \zeta_{2d}^z - \zeta_{2d}^{-z}, \, \text{where} \, -d/2 < z < d/2, (z, 2d)=1 \right\}.
\]
By Galois theory, the set of all conjugates of $\zeta_{2d} - \zeta_{2d}^{-1}$ over $\mathbb{Q}$ is 
\[
\mathcal{S}_d := \left\{ \zeta_{2d}^z - \zeta_{2d}^{-z}, \, \text{where} \, 1 \le z < 2d, (z, 2d)=1 \right\}.
\]
When $1 \le z < 2d, (z, 2d)=1$ and $z$ falls outside the interval $(-d/2, d/2)$, it must be that $d/2 < z < 3d/2$. In this case, we have
\[ -d/2 < d-z < d/2, \quad (d-z, 2d)=1 \]
and 
\[
\zeta_{2d}^{d-z} - \zeta_{2d}^{-(d-z)} = \zeta_{2d}^d \zeta_{2d}^{-z} - \zeta_{2d}^{-d} \zeta_{2d}^z = -\zeta_{2d}^{-z} - (-\zeta_{2d}^z) = \zeta_{2d}^z - \zeta_{2d}^{-z}.
\]

Therefore, we obtain 
\[ R_d = \mathcal{S}_d. \]
Thus $\Omega_d(x) = \prod_{x_i \in R_d} (x-x_i)$ is the minimal polynomial of the algebraic integer $\zeta_{2d} - \zeta_{2d}^{-1}$ over the rational numbers and $\Omega_d(x) \in \mathbb{Z}[x]$ is irreducible over $\mathbb{Q}$ by the characteristic of the minimal polynomial.

Now $\Omega_d(x)$ is monogenic by Lemma \ref{comfieldmonolucas}.

\section{Concluding Remarks and Open Problems}

Theorems \ref{monofibo} and \ref{monolucas} give complete monogenity results for the irreducible factors of $F_n(x)$ when $n$ is odd and of $L_n(x)$ when $n$ is even. 
Numerical experiments for $n \le 100$ indicate that the non-monogenic factors in remaining cases do not have a naive density, suggesting a deeper obstruction that warrants further study. We propose several related problems.

\begin{problem}
	Determine a necessary and sufficient condition on an even integer $n$ such that every irreducible factor of the Fibonacci polynomial $F_n(x)$ is monogenic.
\end{problem}

\begin{problem}
	Determine a necessary and sufficient condition on an odd integer $n$ such that every irreducible factor of the Lucas polynomial $L_n(x)$ is monogenic.
\end{problem}

\begin{problem}
	More generally, given the families $F_n(x)$ and $L_n(x)$ for $n \ge 1$, is there an algorithmic criterion to decide exactly when all of their irreducible factors are monogenic? In particular, can such a criterion be completely described solely in terms of the prime factorization and the parity of $n$?
\end{problem}

{\footnotesize
	
	\bibliographystyle{amsplain}
	\bibliography{monoseq}
	
		\paragraph{Han Chen,\, Weizhe Guo,\, Haojie Hong} 
	~\\
	School of Mathematics and Statistics, Hainan University, No. 58 Renmin Avenue, Haikou 570228, PR, China\\
	E-mail: hchen@hainanu.edu.cn, \,24210701000007@hainanu.edu.cn,\, hjhong@hainanu.edu.cn

}
\end{document}